\documentclass{article}

\usepackage{amsmath}
\usepackage{mathrsfs}

\usepackage{amssymb}
\usepackage{amsthm}
\usepackage{enumerate}
\usepackage{bbm}
\usepackage[
backend=biber,
style=alphabetic,
]{biblatex}
\usepackage{multicol}
\usepackage{amsfonts}
\usepackage[colorlinks=true,citecolor=red,linkcolor=blue]{hyperref}
\usepackage{xfrac}
\usepackage[all]{xy}
\usepackage{tikz-cd}
\newtheorem{thm}{Theorem}[section]
\newtheorem{proposition}[thm]{Proposition}

\newtheorem{lemma}[thm]{Lemma}

\newtheorem{rmk}[thm]{Remark}

\newtheorem*{theorem*}{Theorem}
\newtheorem*{corollary*}{Corollary}
\newtheorem{mainthm}{Theorem}

\newtheorem{definition}[thm]{Definition}

\newcommand{\pp}{\mathfrak{p}}
\newcommand{\qq}{\mathfrak{q}}

\newcommand{\Pp}{\mathfrak{P}}
\newcommand{\Qq}{\mathfrak{Q}}

\renewcommand{\O}{\mathcal {O}}
\newcommand{\F}{\mathbb{F}}

\newcommand{\lv}{\mathcal{L}_{{\rm val}}}

\newcommand{\dif}{\operatorname{Diff}}

\newcommand{\br}{\operatorname{Br}}
\newcommand{\inv}{\operatorname{inv}}
\newcommand{\pol}{\operatorname{Poles}}

\newcommand{\Gal}{\operatorname{Gal}}
\newcommand{\congruente}[3]{#1 \equiv #2 (\bmod\, #3)}

\newcommand{\sqn}[1]{\sqrt[n]{#1}}

\title{Undecidability of infinite towers of Kummer extensions of $\F_p(t)$ }
\author{ Carlos  Mart\'inez-Ranero \\ Javier Utreras}

\begin{document}

\maketitle
\begin{abstract}
  We prove, assuming resolution of singularities in positive characteristic, an analogue of Siegel's theorem on sum of squares in positive characteristic. The method of proof combines techniques from central simple algebras with model theory and builds on work of Anscombe, Dittmann and Fehm.
 As an application, we show that, for each finite field $\F$ of odd characteristic and any positive integer $n$ coprime with the characteristic of $\F$, the first-order theory of the field given by the compositum of the fields generated by adjoining the $n$--th roots of all monic irreducible polynomials in $\F[t]$, of degree divisible by $n$ is undecidable in the language of rings with the variable $t$ as a constant. 
 \footnote{ The first named author was  supported by  Proyecto VRID-Investigación  220.015.024-INV  \\The second named author was supported by Proyecto VRID 2022000419-INI }
 \end{abstract}
 \section{Introduction and main results}

In \cite{robinson49}, J. Robinson posed the question of determining which algebraic extensions ( possibly infinite ones) of $\mathbb{Q}$ have a decidable first--order theory. A major breakthrough came a decade later, when J. Robinson herself demonstrated in \cite{robinson59} that all number fields are undecidable. The case of infinite algebraic extensions, however, has proven to be significantly more challenging, and progress has been slower. Nevertheless,  a well-developed set of tools and techniques for proving the undecidability of the first-order theory of infinite algebraic extensions of $\mathbb{Q}$ has emerged. These methods have been employed by various authors to construct families of undecidable infinite algebraic extensions. (see e.g.  \cite{videlaelliptic}, \cite{videlaprop},\cite{fukuzaki}, \cite{sasha},\cite{MUV},\cite{caleb},\cite{caleb2},\cite{vv}, \cite{ms}).

In this article, we investigate the analogous question in positive characteristic, with two primary goals: first, to construct new families of infinite algebraic extensions of $\mathbb{F}_p(t)$ with undecidable first-order theories; and second, to advance the theory in positive characteristic, bringing it closer to the level attained in characteristic $0$.

Before stating our main results, we need to introduce some notation. We work under the following assumption, which would follow from resolution of singularities in positive characteristic (see  \cite[Proposition 2.3]{ADFaxiom} for details).

\textbf{(R4)} Every large field $K$ is existentially closed in every extension $F/K$ for which there exists a valuation $v$ on $F/K$ with residue field $\F_v=K$. 

Let $p$ be an odd rational prime, let $q$ be a power of $p$ and let $n$ be a natural number coprime with $p$. Let $K:=\F_q(t)$ and $\O_K:=\F_q[t]$. Let $P_n^+$ denote the set of irreducible monic polynomials of $\O_K$ of degree a multiple of $n$. Suppose that there is a primitive $n$-th root of unity $\zeta_n\in K$. 
Let $\{ p_m(t): m\in\mathbb{N}\}$ be an enumeration of $P^+_n$. We define recursively a tower of fields as follows: $K_0=K$, and $K_{m+1}:=K_m\left (\sqn{p_m(t)}\right )$. Finally, let $K_{\rm inf}:=\bigcup\limits_{m=1}^\infty K_m.$
Our first main theorem is as follows.
\begin{mainthm}[see Theorem \ref{thm:kummer}] \label{mainthm:kummer}
   Assume (R4). The first-order $\mathcal{L}_{{\rm ring},t}$-theory of the field $K_{\rm inf}$ is undecidable.
\end{mainthm}

The  proof follows the usual approach. First, we apply \cite[Theorem B]{MSU} to define the ring of integers $\O_{K_{\rm inf}}$. A standard reduction argument then shows that the undecidability of $K_{\rm inf}$ can be deduced from the undecidability of $\O_{K_{\rm inf}}$.   Next, using \cite[Theorem C]{MSU} we further reduce the problem to constructing a parametrized family of definable sets which contains sets of arbitrarily large finite cardinalities. Inspired by J. Robinson's work on totally real extensions, we observe that such a family  can be constructed using an analogue of Siegel's theorem on sum of squares. 

Recall that Siegel's Theorem states that on a number field every totally positive element can be expressed as the sum of four squares (thus roughly speaking the classical Pythagoras number is $\leq 4$). For a global field $F$ and prime  $\pp$ of $F$, we define the  $\pp$-Pythagoras number (see Section 3 for the definition). Intuitively, while the Pythagoras number gives some insight into the set of totally positive elements, the $\pp$-Pythagoras number provides information about the set of totally $\pp$-integral elements.

Our next main result is the following.
\begin{mainthm}[see Theorem \ref{thm:siegel}] \label{mainthm:siegel}
   Assume (R4). Let $K$ a global function field of odd characteristic and $\pp$ a place of $K$, there exists $N$ such that $\pi_\pp(L)\leq N$ for any finite separable extension $L/K$.
\end{mainthm}

 Recently,  Anscombe,  Fehm, and  Dittman  \cite{ADFsiegel} obtained a $p$--adic analogue of Siegel's theorem. Their  proof combines techniques of model theory with central simple algebras. We adapt their proof to positive characteristic. A key difficulty in this adaptation comes from the lack of an axiomatization of the first order theory of the Laurent field $\F_q((t))$ . It is worth pointing out that the search for such an axiomatization is one of the main open problems in the area. Nonetheless, in a separate  work,  Anscombe, Dittman and Fehm \cite{ADFaxiom} were able to obtain, assuming resolution of singularities in positive characteristic, an axiomatization of the universal theory of $\F_q((t))$. This turn out to be  strong enough for our purposes. 

The structure of the paper is as follows. In Section 2 we introduce the Kochen ring and prove a characterization of its integral closure. In Section 3 we define the $\pp$-Pythagoras number. In Section 4 we develop the necessary model-theoretic tools. We prove Theorem \ref{mainthm:siegel} in Section 5. Finally, in Section 6 we prove Theorem \ref{mainthm:kummer}.

\section{The Kochen ring}
We fix, for the rest of the section, a triple  $(K,\pp,t_\pp)$ where $K$ is a global function field, $\pp$ a prime of $K$ and $t_\pp$ an uniformizer. Let $v_\pp$ denote the associated discrete valuation of rank $1$ with value group $\mathbb{Z}$.  Let $q$ denote the size of the residue field $\F_\pp$.
For a field extension $F/K$ with $\Pp$ a prime of $F$ above $\pp$, we define the \emph{ramification} number $e(\Pp|\pp):=v_\Pp(t_\pp)$ and the \emph{residue degree} $f(\Pp|\pp):=[\F_\Pp: \F_\pp].$ We denote by $\mathcal{S}{(F)}$ the set of primes of $F$, $\mathcal{S}_\pp^*(F)$ the set of primes of $F$ above $\pp$  and by $\mathcal{S}^{(1,1)}_\pp(F)$ the set of primes of $F$ above $\pp$ with $e(\Pp|\pp)=1$ and $f(\Pp|\pp)=1.$ 

We consider the holomorphy ring
$$
R_\pp(F):=\bigcap\limits_{\Pp\in\mathcal{S}^{(1,1)}_\pp(F)}\O_\Pp
$$
 where $\O_\Pp$ denotes the valuation ring of $\Pp$.
 \begin{rmk}
     If $\mathcal{S}^{(1,1)}_\pp(F)=\emptyset$, then we set $R_\pp(F)=F.$
 \end{rmk}

 Let $$\Gamma_\pp(F):=\left \{\frac{a}{1+t_\pp b}: a,b \in \O_\pp[\gamma(F\setminus \pol{(\gamma)})], 1+t_\pp b\ne 0 \right \}$$ denote the Kochen ring 
 where $\gamma(x):=\frac{1}{t_\pp}\cdot\frac{x^q-x}{(x^q-x)^2-1}$ is the Kochen operator. 
 
We shall show that the holomorphy ring  above defined is equal to the integral closure of the Kochen ring. To do this, we need a few Lemmas. It is worth pointing out that most of the results of this section are well known in the case of formally $p$--adic fields (see   \cite[Chapter 6]{prestelroquette}). 

 \begin{lemma}\label{valuaciones}
Let $F/K$ be a field extension and let $\Pp\in \mathcal{S}(F)$,  $a\in F\setminus \pol(\gamma)$ be given. Let $\beta(x):=\frac{x^q-x}{(x^q-x)^2-1}$. Then 
    \begin{enumerate}[(i)]
        \item If $v_{\Pp}(a)>0,$ then $v_\Pp(\beta(a))=v_\Pp(a)>0.$
        \item If $v_\Pp(a)<0,$ then $v_\Pp(\beta(a))=-qv_\Pp(a)>0.$
        \item If $v_\Pp(a)=0$ and $v_\Pp(a^q-a)>0,$ then $v_\Pp(\beta(a))=v_\Pp(a^q-a)>0.$
        \item If $v_\Pp(a)=0$ and $v_\Pp(a^q-a)=0,$ then $v_\Pp(\beta(a))\leq 0.$
    \end{enumerate}
\end{lemma}
\begin{proof}
    These are just straightforward computations.
\end{proof}

\begin{lemma}\label{lema1.2}

Let $F/K$ be a field extension and let $\Pp\in\mathcal{S}_\pp^*$. 
Then $\Pp\in \mathcal{S}^{(1,1)}_\pp(F)$  if and only if  $\pol(\gamma)=\emptyset$ and $\gamma(a)\geq 0$ for all $a\in F$. 
    
\end{lemma}
\begin{proof}
    Let $\Pp\in \mathcal{S}^{(1,1)}_\pp(F)$ be given. Let $a\in F$. Using $v_\Pp(t_\pp)=1$, we see that $v_\Pp(a)\geq 0$ for any $a$ satisfying clauses (i)-(iii) of the previous Lemma. Now we show that $a$ can not satisfy clause (iv) of Lemma \ref{valuaciones}. Indeed, we deduce from $f(\Pp|\pp)=1$ that $\mathbb{F}_\Pp\cong \F_q$. Hence, $v_\Pp(a^q-a)>0$ for any $a\in F$ with $v_\Pp(a)= 0$. Thus $\pol(\gamma)=\emptyset$ and $v_\Pp(a)\geq 0,$ for any $a\in F$.
    
Conversely, let $\Pp\in\mathcal{S}_\pp^*$ be given such that $\gamma(a)\geq 0$ for all $a\in F.$ First we show that $e(\Pp|\pp)=1.$ To see this, let $a\in F$ be given with $v_\Pp(a)>0$. We see, from clause (i) of Lemma \ref{valuaciones}, that
$0\leq v_\Pp(\gamma(a))=v_\Pp(t_\pp^{-1}\beta(a))=-v_\Pp(t_\pp)+v_\Pp(a)$. Hence, $v_\Pp(a)\geq v_\Pp(t_\pp)$ for any $a\in F$. Since $t_\pp$ has minimal positive value, we have $e(\Pp|\pp)=1$.  Next we show that $f(\Pp|\pp)=1$. To show this, let $a\in F$ with $v_\Pp(a)=0$. If $v_\Pp(a^q-a)=0,$ then $v_\Pp(\gamma(a))=v_\Pp(t_\pp^{-1}\beta(a))=-1+v_\Pp(\beta(a))<0$ (by  clause (iv) of lemma \ref{valuaciones}), which is a contradiction.

Thus, $v_\Pp(a^q-a)>0$. Since $a$ was arbitrary we conclude that $\mathbb{F}_\Pp\cong \F_q$.
\end{proof}

\begin{lemma}\label{characval}
    Let $F/K$ be a field extension.  If $\Pp$ is a prime of $F$ such that $t_\pp\in \Pp
    $ and  $\O_\pp[\gamma(F\setminus \pol(\gamma))]\subseteq \O_\Pp$, then $\Pp\in \mathcal{S}^{(1,1)}_\pp(F)$.
\end{lemma}
\begin{proof}
    Let $\Pp$ be a prime of $F$ as in the statement of the Lemma.  First we show that $\Pp | \pp$. Since $\O_\pp\subseteq \O_\Pp,$ we have that $\O_\pp\subseteq \O_\Pp\cap K$. On the other hand, let $a\in K\cap \O_\pp$. Using that $t_\pp$ is a uniformizer of $\pp$ we can find $u\in \O_\pp^\times$ and $n\in\mathbb{Z}$ such that $a=ut^n$. Notice that $n\geq 0$ since $t_\pp\in \Pp$ and thus in particular $t_\pp^{-1}\notin \O_\Pp$. We deduce that $a\in \O_\pp$ and hence $\O_\Pp\cap K =\O_\pp.$  We next show that $\Pp\in \mathcal{S}^{(1,1)}_\pp(F).$ Using lemma \ref{lema1.2}, it is sufficient to show that $\pol(\gamma)=\emptyset$. Indeed, since $\gamma$ is continuous (with  respect to the topology induced by $v_\pp$) and $\O_\Pp$ is a closed subset of $F$, we have  $\gamma^{-1}(\O_\Pp)=F\setminus \pol(\gamma)$ is closed in $F$. Using the fact that $\pol(\gamma)$ is a finite set  and the topology is not discrete, we obtain that $\pol(\gamma)=\emptyset$. This concludes the proof of the Lemma. 
\end{proof}

We are now ready to prove the main Theorem of the section.

 \begin{thm}\label{thmkochenring}
     The integral closure of the Kochen ring  $\Gamma_\pp(F)$ is equal to $R_\pp(F)$. 
 \end{thm}
 \begin{proof}
     Notice that if  $t_\pp$ is a unit in  $\O_\pp[\gamma(F)\setminus \pol(\gamma)]$, then $1+t_\pp\O_\pp[\gamma(F)\setminus \pol(\gamma)]=\O_\pp[\gamma(F)\setminus \pol(\gamma)]$. Thus, $\Gamma_\pp(F)$ is the fraction field of $\O_\pp[\gamma(F\setminus \pol(\gamma)].$ Applying Merkel's Lemma (see appendix in \cite{prestelroquette}), we obtain that  $\Gamma_\pp(F)=F$. 
     
     Next suppose $t_\pp$ is not unit in  $\O_\pp[\gamma(F)\setminus \pol(\gamma)]$. Define   $$\mathbb{V}:=\{\Pp\in \mathcal{S}(F): \Gamma_\pp(F)\subseteq \O_\Pp, \ \Gamma_\pp(F)\cap \Pp \ \rm{is \ a \ maximal \ ideal\ of} \ \Gamma_\pp(F)\}.$$ Recall that the integral closure of $\Gamma_\pp(F)$ is equal to $\bigcap\limits_{\Pp\in \mathbb{V}}\O_\Pp$ (see  \cite[Theorem 3.1.3]{englerprestel}). Thus, it suffices to show that $\mathbb{V}=\mathcal{S}^{(1,1)}_\pp(F)$.\\
     
  Let $\Pp\in \mathcal{S}^{(1,1)}_\pp(F)$ be given.  We deduce, from Lemma \ref{lema1.2}, that  $\O_\pp[\gamma(F)]\subseteq \O_\Pp$ and thus also $\Gamma_\pp(F)\subseteq \O_\Pp$ (since $v_\Pp(t_\pp)=1).$ We claim that  $\mathfrak{m}:=\Pp\cap \Gamma_\pp(F)$ is a maximal ideal. To verify this, notice that $\Gamma_\pp(F)/\mathfrak{m}$ is a field since it is a subring of $\F_\Pp=\F_q.$ Therefore $\mathfrak{m}$ is a maximal ideal and thus $\Pp\in \mathbb{V}.$

 On the other hand, let $\Pp\in \mathbb{V}$ be given. In view of the previous Lemma, it is sufficient to show that $t_\pp\in \Pp
    $ and  $\O_\pp[\gamma(F\setminus \pol(\gamma))]\subseteq \O_\Pp$.  Observe that $\O_\pp[\gamma(F\setminus \pol(\gamma))]\subseteq \Gamma_\pp(F)\subseteq \O_\Pp$. We claim that all maximal ideals of $\Gamma_\pp(F)$ contain $t_\pp$. Aiming towards a contradiction, suppose that there is a maximal ideal $\mathfrak{m}$ which does not contain $t_\pp$. We deduce that $t_\pp$ is a unit modulo $\mathfrak{m}$. Thus, there is an $a\in \Gamma_\pp(F)$  such that $\congruente{at_\pp}{1}{\mathfrak{m}}.$ Write $a:=\frac{b}{1+t_\pp c}$ where $b, c\in \O_\pp[\gamma(F\setminus \pol(\gamma))].$ Then $\congruente{1+t_\pp(c-b)}{0}{\mathfrak{m}}.$ Notice that $1+t_\pp(c-b)\ne 0$ as $t_\pp$ is not a unit in $\O_\pp[\gamma(F\setminus \pol(\gamma))]$ by assumption. Therefore $\congruente{1}{0}{\mathfrak{m}}$ which is a contradiction. Thus $t_\pp\in \mathfrak{m}=\Gamma_\pp(F)\cap\Pp$. We conclude the proof of the Theorem.
    \end{proof}
    
The following Lemma will be needed in the Proposition \ref{keyprop}.
\begin{lemma}\label{keylemma}
    Let $F/K$ be a finite separable field extension and let $\Pp\in \mathcal{S}^*_\pp(F)\setminus\mathcal{S}^{(1,1)}_\pp(F)$ be given. Then there exists $a\in F$ such that $a$ is not a pole of $\gamma$ and $v_\Pp(\beta(a))\leq 0.$
\end{lemma}
\begin{proof}
    First suppose $e(\Pp|\pp)>1$. Let $t_\Pp\in \O_\Pp$ be a uniformizer. Notice that $2=2v_\Pp(t_\Pp)\leq e(\Pp|\pp)=v_\Pp(t_\pp).$ Using  Lemma \ref{valuaciones} clause (i) we see that $v_\Pp(t_\pp\gamma(t_\Pp))=v_\Pp(\beta(t_\Pp))=-e(\Pp|\pp)\leq -2.$ Thus $a=t_\Pp$ is as required. Next suppose $e(\Pp|\pp)=1$ and 
    $f(\Pp|\pp)>1$. Let $a\in F\setminus \pol(\gamma)$ be such that $v_\Pp(a)<0$ ($a$ exists by Lemma \ref{lema1.2}). Using  $e(\Pp|\pp)=1$ we see that $v_\Pp(a)\geq 0$ for any $a\in F\setminus \pol(\gamma)$ which satisfies any of the hypothesis (i)--(iii) of Lemma \ref{valuaciones}.  Thus $a$ satisfies $v_\Pp(a)=0, v_\Pp(a^q-q)=0$, and it is also not a pole of $\gamma$. Hence $v_\Pp(\beta(a))\leq 0 $ which is what we wanted to prove. 
\end{proof}


    \section{The $\pp$-Pythagoras number}	
We recall the definition of $\pp$-Pythagoras number introduced by S. Anscombe, P. Dittman and A. Fehm (see \cite{ADFsiegel}).

Let $F/K$ be a field extension. For any $g\in\O_\pp[x_1,\dots,x_n]$, we write $$ R_{\pp,g,t_\pp}(F):=\left \{\frac{a}{1+t_\pp b}: a,b\in g(\gamma(F\setminus \pol{(\gamma)}),\dots,\gamma(F\setminus \pol{(\gamma}))), 1+t_\pp b\ne0\right \},$$
and for $n\geq 1,$ we set $R_{\pp,g,t_\pp,n}(F)$ 

$$\left \{x\in F: x^m+a_{m-1}x^{m-1}+\dots+a_0=0, 1\leq m\leq n, a_{m-1},\dots,a_0\in R_{\pp,g,t_\pp}(F) \right \}.$$

Let $\mathcal{P}_{\pp,n}$ denote the set of polynomials $g\in\O_\pp[x_1,\dots,x_n]$ of degree and height at most $n$. Let
$$ R_{\pp,n}(F):=\bigcup\limits_{t_\pp}\bigcup\limits_{g\in \mathcal{P}_{\pp,n}}R_{\pp,g,t_\pp,n}(F)$$
where $t_\pp$ varies over the (finite) set of uniformizers of $\pp$ of minimal height. Note that $\left ( R_{\pp,n}(F)\right )_{n\geq 1} $ is an increasing sequence and $R_\pp(F)=\bigcup\limits_n R_{\pp,n}(F) $. 

The $\pp$-{\bf Pythagoras\ number} $\pi_\pp(F)$ is the minimum $n\in \mathbb{N}\cup\{\infty\}$ such that $$R_\pp(F)=\bigcup\limits_{k\leq n} R_{\pp,k}(F).$$ 

\section{Diophantine families}
In this section we introduce the model theoretic tools that we need to prove the analogue of Siegel's Theorem.

We first recall the notion of diophantine family introduced by S. Anscombe, P. Dittman and A. Fehm (see \cite{ADFsiegel}). An $n$-dimensional diophantine family over $K$ is a map $D$ from the class of field extensions $F/K$ to sets which is given by finitely many polynomials $f_1,\dots,f_r\in K[X_1,\dots,X_n,Y_1,\dots,Y_m],$ for some $m$, in the sense that 
$$
D(F)=\{x\in F^n: \exists y\in F^m [f_1(x,y)=0,\dots,f_r(x,y=0]\}
$$
for every extension $F/K$. From the point of view of model theory, $D$ is given by an existential formula $\varphi(x_1,\dots,x_n)$ in the language of rings with free variables $x_1,\dots,x_n$ and parameters from $K$, and $D(F)$ is the realization of $\varphi$ in $F$, i.e., $D(F)=\{x\in F^n: F\models \varphi(a)\}$.

Diophantine families are closed under unions, intersections, cartesian products and images of rational functions, in the sense that there is a diophantine family $D$, such that for any field extension $F/K,$ $D(F)$ realizes the corresponding operation (see  \cite[Section 3]{ADFsiegel} for details).

We denote by $\mathcal{L}_{\rm rings}=\{+,-,\cdot,0,1\}$ the language of rings, by $\mathcal{L}_{\rm val}=\mathcal{L}_{\rm rings}\cup \{\O\}$ the language of valued fields, where $\O$ is a unary predicate symbol, and by $\mathcal{L}_{\rm val}(\varpi)=\mathcal{L}_{\rm val}\cup\{\varphi)$, where $\varpi$ is a constant symbol.   We view a valued field $(K,v)$ as an $\mathcal{L}_{\rm val}$-structure by interpreting $\O$ as the valued ring $\O_v$. A uniformizer $\pi$ of $K$ is an element of minimal positive value. We view $(K,v,\pi)$ as an  $\mathcal{L}_{\rm val}(\varpi)$-structure where $\varpi$ is interpreted as $\pi$.

Finally, given a field extension $K/C$ we denote by $\mathcal{L}_{\rm val}(\varpi,C)$ the expansion of $\mathcal{L}_{\rm val}(\varpi)$ where we add a constant symbol for each element of $C$, and $(K,v,\pi,C)$ denotes the $\mathcal{L}_{\rm val}(\varpi,C)$-structure expanding $(K,v,\pi)$ in the usual way, i.e., the constant symbol $c_x$ is interpreted by $x$.
As mentioned in the introduction to this article, we will need the following assumption, which follows from resolution of singularities in positive characteristic (see  \cite[Proposition 2.3]{ADFaxiom} for details).

 (R4)\label{R4}  Every  large  field $K$ is existentially closed in every extension $F/K$ for which there exists a valuation $v$ on $F/K$ with residue field $\F_v=K.$ 
\begin{definition}
    We denote by $T$ the $\mathcal{L}_{\rm val}(\varpi)$-theory of equicharacteristic henselian valued fields with distinguished uniformizer.
\end{definition}
\begin{proposition}{\rm \cite[Proposition 4.11]{ADFaxiom}}\label{propADFaxiom}
    Assume {\rm (R4)}. Let $(C, u)$ be an equicharacteristic $\mathbb{Z}$-valued field with distinguished uniformizer $\pi$ such that $\O_u$ is excellent. Let $(K, v, \pi)$, $(L, w, \pi) \models T$ be extensions of $(C, u, \pi)$ such that $\F_v/\F_u$ and $\F_w/\F_u$ are separable. If ${\rm Th}_\exists(\F_v,\F_u) \subseteq {\rm Th}_\exists (\F_w,\F_u)$, then ${\rm Th}_\exists(K,v,\pi,C)\subseteq {\rm Th}_\exists(L,w,\pi,C)$.
\end{proposition}

\begin{rmk}
    If $(C,u)$ is a global function field then the valuation ring $\O_u$ is excellent.
\end{rmk}

\begin{proposition}\label{propdioph} Assume ${\rm (R4)}$.
    Let $D$ be a 1-dimensional Diophantine family and $\mathcal{K}$ a class of extensions of $K$. If
    \begin{enumerate}[(i)]
        \item $D(L)=R_\pp(L)$ for any $L\in \mathcal{K}$, and
        \item $D(K_\pp)\subseteq \O_{\hat \pp},$
    \end{enumerate}
    then there exists $N$ such that $\pi_\pp(L)\leq N$ for all $L\in\mathcal{K}.$
\end{proposition}
\begin{proof}
 We claim that $D(F)\subseteq R_\pp(F)$ for any field extension $F/K$. To do this,  let $F$ be any extension of $K$. Suppose first that  $\mathcal{S}^{(1,1)}_\pp(F)=\emptyset$. Then $D(F)\subseteq F=\Gamma_\pp(F)=R_\pp(F)$ (see Theorem \ref{thmkochenring}). Now assume that  $\mathcal{S}^{(1,1)}_\pp(F)\ne\emptyset$, and let $\Pp\in \mathcal{S}^{(1,1)}_\pp(F) $ be given.  Consider the $\mathcal{L}_{\rm val}(\varpi)$-structures  $(F^h,\Pp^h, t_\pp)$ and $(K_\pp,\hat{\pp}, t_\pp)$ where the former  denotes the henselization of $(F,\Pp, t_\pp)$ and the latter  the $\pp$-adic completion of $K$.  It follows from our assumptions that $\F_{\pp^h}\cong \F_\pp\cong \F_{\hat \Pp} $ and both are $\mathbb{Z}$-valuations. Applying  Proposition \ref{propADFaxiom} to the equicharacteristic $\mathbb{Z}$-valued field $(K,v_\pp)$ with distinguished uniformizer $t_\pp$, we obtain  that  the universal $\mathcal{L}_{\rm val}(\varpi)$-theories  ${\rm Th}_\forall(F^h,\Pp^h, K)$ and ${\rm Th}_\forall(K_\pp, \hat{\pp},K)$ are equal.  Let $\varphi(x,\overline{c})$ be the existential formula defining $D$ over $K$. By clause (ii) we have that the universal $\lv(\varpi,K)$-formula $\Theta(x): \forall x [\neg \varphi(x,\overline{c})\vee x\in \O_v]$ holds in $(K_\Pp,v_\Pp)$, and so also holds in the structure $(F^h,v_{\Pp^h})$.  Thus, $D(F)\subseteq \O_{\Pp^h}\cap F=\O_\Pp.$ Since $R_\pp(F)=\bigcup_{n=1}^\infty R_{\pp,n},$ then there is an $N$ such that $D(F)\subseteq R_{\pp,N}$ for all field extensions $F/K$ (by \cite[Proposition 3.9]{ADFsiegel}). Notice that for $L\in \mathcal{K}$
$R_\pp(L)=D(L)\subseteq R_{\pp,N}$. Therefore, $\pi_\pp(L)\leq N$. This concludes the proof of the proposition. \end{proof}
\section{Central simple algebras}
The goal of this section is to  find a $1$--dimensional diophantine family satisfying the assumptions of Proposition \ref{propdioph}. To do this, we will use some techniques from the theory of central simple algebras. This method was developed by P. Dittmann \cite{Dittmann} building on work of B. Poonen \cite{Poonen}. For further details on central simple algebras, the Brauer group and  associated concepts see \cite{GilleSzamuely}.

Let $A$ a central simple algebra of prime degree $\ell$ over a field $F$. We let
$$ S_A(F):=\{Trd(x): x\in A, Nrd(x)=1\}\subseteq F,$$
where $Trd$ and $Nrd$ are the reduced trace and reduced norm, respectively. We define

$$T_A(F):=\begin{cases}
    S_A(F) & {\rm if}\ \ell>2 \\
    S_A(F)-S_A(F) & {\rm if}\ \ell=2
\end{cases}$$
\begin{lemma}{\rm \cite[Lemma 2.12]{Dittmann}}
    $S_A$ and $T_A$ are $1$-dimensional diophantine families over $F$.
\end{lemma}

\begin{proposition}{\rm \cite[Proposition 2.9]{Dittmann}}\label{global}
    Let $L$ be a global function field and $A$ a central simple algebra over $L$ of prime degree $\ell$. Then $$
    T_A(L)=\bigcap\limits_{\pp\in \Delta(A/L)}\O_\pp 
    $$ 
    where $\Delta(A/L)$ denotes the  finite set of primes of $L$ such that $A$ does not split over $L_\pp.$
\end{proposition}

\begin{proposition}\label{local}
    Let $F$ be a local field of positive characteristic and let $A$ be a central simple algebra over $F$ of prime degree $\ell$. If $A$ is not split then $T_A(F)=\O_F.$
 \end{proposition}
 \begin{proof}
Using Wedderburn's Theorem we see that a non-split central simple algebra of prime degree is a division algebra. Then the result follows from \cite[Proposition 2.6]{Dittmann}.
     
 \end{proof}
The following is essentially \cite[Proposition 5.4]{ADFsiegel}.
 \begin{proposition}\label{brauer}
     For every prime number $\ell$ there exist central simple algebras $A,B$ of degree $\ell$ over $K$ such that
     \begin{enumerate}
         \item neither of them split over $K_\pp$,
         \item for every place $\mathfrak{q}\ne \pp$ of $K$, at least one of them split over $K_{\mathfrak{q}}$.
    \end{enumerate}
 \end{proposition}
 \begin{proof}
     Let $\br(K)$ denote the Brauer group of $K$ and let $\inv_{K_\pp}: \br(K)\to \mathbb{Q}/\mathbb{Z}$ denote the local Hasse invariant. It follows from  \cite[Proposition 6.3.9]{GilleSzamuely} that $\inv_{K_\qq}$ is an isomorphism and thus an algebra splits over $K_\qq$ if and only if its local invariant is $0$. Fix $\qq_1\ne \qq_2$  primes of $K$ both different from $\pp$. Let $[A_{\pp}]:=\inv^{-1}_{K_{\pp}}(\frac{1}{\ell}), [A_{\qq_1}]:=\inv^{-1}_{K_{\qq_1}}(\frac{\ell-1}{\ell})$ and $[A_{\qq_2}]:=\inv^{-1}_{K_{\qq_2}}(\frac{\ell-1}{\ell})$.     Using Hasse's theorem (\cite[Corollary 6.5.4]{GilleSzamuely}) we have an exact sequence 
     \[
\xymatrix{
  0 \ar[r] &  \br(K)  \ar[r]^-{\sum -\otimes K_\qq} & \bigoplus\limits_{\mathfrak{q}\in \mathcal{S}_K}\br(K_\qq)  \ar[r]^-{\sum\inv_{K_\qq}} & \mathbb{Q}/\mathbb{Z}  \ar[r] & 0 
}
\]

Since  $[A_\pp]+[A_{\qq_1}]$ and $[A_\pp]+[A_{\qq_2}]$  belong to the kernel of ${\sum\inv_{K_\qq}}$, then there are unique equivalence classes $[A]$ and $[B]$ such that  $\sum [A]\otimes K_\qq=[A_\pp]+[A_{\qq_1}]$ and $\sum [B]\otimes K_\qq=[A_\pp]+[A_{\qq_2}]$. We infer from this that both $[A], [B]$ have period $\ell.$ Using \cite[Remark 6.5.5]{GilleSzamuely}, we have that the period is equal to the index. Thus, if  $A$ and $B$  denote the unique division algebras in the class of $[A]$ and $[B]$, respectively,  then they have degree $\ell.$  This concludes the proof of the proposition.

 \end{proof}

For central simple algebras $A,B$ over $K$ and an extension $F/K$ we define 

$$
D_{\pp,t_\pp,A,B}(F):=\left \{\frac{x}{1+t_\pp w y}: x,y\in T_A(F)+T_B(F), w\in \gamma(F\setminus \pol(\gamma)), 1+t_\pp wy\ne 0 \right \}
$$
\begin{proposition}\label{local-global}
    If $A,B$ are $K$-algebras as in Proposition \ref{brauer}, then 
    \begin{enumerate}[i.]
        \item $T_A(K_\pp)+T_B(K_\pp)=\O_{\hat{\pp}}$,
        \item for all finite separable extensions $L/K$,
        $$
        T_A(L)+T_B(L)\supseteq \bigcap\limits_{\Pp\in \mathcal{S}^*_\pp}\O_\Pp.
        $$
    \end{enumerate}
\end{proposition}
\begin{proof}
    For the first clause, we have that $T_A(K_\pp)+T_B(K_\pp)=\O_{\hat{\pp}}+\O_{\hat{\pp}}=\O_{\hat{\pp}}$ by Proposition \ref{local}. Next, let $L/K$ be a finite separable extension. Using Proposition \ref{global} we obtain $$T_A(L)+T_B(L)=\bigcap\limits_{\Qq\in \Delta(A/L)} \O_\Qq+\bigcap\limits_{\Qq\in \Delta(B/L)} \O_\Qq$$ 
    We claim that $$\bigcap\limits_{\Qq\in \Delta(A/L)} \O_\Qq+\bigcap\limits_{\Qq\in \Delta(B/L)} \O_\Qq=\bigcap\limits_{\Qq\in \Delta(A/L)\cap \Delta(B/L)} \O_\Qq. $$
    To see this, notice that the left hand side is clearly contained in the right hand side. Thus we need to show the other inclusion. Let $x\in \bigcap\limits_{\Qq\in \Delta(A/L)\cap \Delta(B/L)} \O_\Qq$ be given. Using the weak approximation property, choose $y$ such that 
    $$
    v_\Qq(y)\geq 0 \ {\rm for} \ \Qq\in \Delta(B/L)
    $$
    and 
    $$
    v_\Qq(x-y)\geq 0 \  {\rm for} \ \Qq\in \Delta(A/L)\setminus \Delta(B/L)
    $$
    This is possible due to our choice of $A$ and $B$. Observe that $x-y\in \bigcap\limits_{\Qq\in \Delta(A/L)} \O_\Qq$ and $y\in \bigcap\limits_{\Qq\in \Delta(B/L)} \O_\Qq$. This concludes the proof of the proposition. 
\end{proof}

\begin{proposition}\label{keyproplocal}
    If $A,B$ are algebras as in Proposition \ref{brauer}, then $D_{\pp,t_\pp,A,B}(K_\pp)\subseteq \O_{\hat\pp}$
\end{proposition}
\begin{proof}
    Using Proposition \ref{local} and Lemma \ref{lema1.2} we have that  $T_A(K_\pp)+T_B(K_\pp), \gamma(K_\pp)\subseteq \O_{\hat\pp}$ and $1+t_\pp \O_{\hat\pp}\subseteq  \O_{\hat\pp}^\times.$ Hence $D_{\pp,t_\pp,A,B}(K_\pp)\subseteq \O_{\hat\pp}$ as required.
\end{proof}
\begin{proposition} \label{keyprop}
    If $A,B$ are $K$-algebras as in Proposition \ref{brauer}, then 
    $$
    D_{\pp,t_\pp,A,B}(L)=R_\pp(L),
    $$
    for any finite separable extension $L/K$.
\end{proposition}
\begin{proof}
    First we show that $D_{\pp,t_\pp,A,B}(L)\subseteq R_\pp(L)$. We proceed by cases. On one hand, if $\mathcal{S}^{(1,1)}_\pp(L)=\emptyset$, then $R_\pp(L)=L$. Thus, clearly $D_{\pp,t_\pp,A,B}(L)\subseteq L$. On the other hand, if $\mathcal{S}^{(1,1)}_\pp(L)\ne\emptyset$, then $L_\Pp=K_\pp$. Thus $D_{\pp,t_\pp,A,B}(L_\Pp)\subseteq \O_{L_\Pp}$ by the previous Proposition. Hence
    $$D_{\pp,t_\pp,A,B}(L)\subseteq \bigcap\limits_{\Pp\in \mathcal{S}^{(1,1)}_\pp(L)}\O_{L_\Pp}\cap L= \bigcap\limits_{\Pp\in \mathcal{S}^{(1,1)}_\pp(L)}\O_{\Pp}=R_\pp(L),$$ which is what we wanted to prove. 

    To prove the other inclusion, let $r\in R_\pp(L)\setminus\{0\}$ be given. Fix  enumerations $\Pp_1,\dots,\Pp_k$ and $\mathfrak{Q}_1,\dots,\mathfrak{Q}_\ell$ of $\mathcal{S}^{(1,1)}_\pp(L)$ and $\mathcal{S}_\pp^*\setminus\mathcal{S}^{(1,1)}_\pp(L)$, respectively. Here  we set $k=0$  ($\ell=0$) if  $\mathcal{S}^{(1,1)}_\pp(L)=\emptyset$ ($\mathcal{S}_\pp^*\setminus\mathcal{S}^{(1,1)}_\pp(L)=\emptyset)$.  First suppose both $k,\ell$ are nonzero. Using Lemma \ref{keylemma}, choose, for each $i\in \{1,\dots,\ell\}$, $z_i\in L$ such that
    $$
    v_{\mathfrak{Q}_i}((t_\pp\gamma(z_i))^{-1})\geq 0.
    $$
Using weak approximation and the continuity of rational functions (with respect to the topology induced by the valuations), pick a $z\in L$ such that 
$$
v_{\mathfrak{Q}_i}((t_\pp\gamma(z))^{-1})\geq 0, \ \ \ {\rm for} \ i\in\{1,\dots,\ell\}. 
$$
Using weak approximation again, choose $y\in L$ such that 
\begin{equation}\label{eq1}
    v_{\mathfrak{Q}_i}((t_\pp\gamma(z))^{-1}+y)\geq \max\{0,-v_{\mathfrak{Q}_i}(rt_\pp\gamma(z))\},\ \ \ i\in \{1,\dots,\ell\}
\end{equation}

and 

\begin{equation}\label{eq2}
v_{\Pp_i}(y)\geq 0,\ \ \ i\in\{1,\dots,k\}.
\end{equation}

Using clauses (\ref{eq1}) and (\ref{eq2}) we obtain that $y\in \bigcap\limits_{\Pp\in \mathcal{S}_\pp^*} \O_\Pp\subseteq T_A(L)+T_B(L)$. Let $x:=r(1+t_\pp\gamma(z)y$. It follows from our choice of $z$ and $y$ that $v_{\mathfrak{Q}_i}(x)\geq 0$ for $i\in\{1,\dots,\ell\}.$  Notice that $r,t_\pp, \gamma(z), y\in \bigcap\limits_{i=1}^k \O_{\Pp_i}.$ Thus $x\in \bigcap\limits_{\Pp\in \mathcal{S}_\pp^*} \O_\Pp$. Since $$\bigcap\limits_{\Pp\in \mathcal{S}_\pp^*} \O_\Pp\subseteq T_A(L)+T_B(L), $$ then by Proposition \ref{global}, we have that $r=x(1+t_\pp\gamma(z)y)^{-1}\in D_{\pp,t_pp,A,B}$ as required. Next suppose $k=0$. In this case, choose $y$ satisfying clause (\ref{eq1}) and the proof goes as before.  Finally suppose $\ell=0.$ Here we set $x:=r$ (which belongs to $T_A(L)+T_B(L)$ by assumption), $y=1$ and $z=0$. We thus have $r=\frac{x}{1+t_\pp \gamma(0)y}\in D_{\pp,t_\pp,A,B}(L).$ This concludes the proof of the proposition.
\end{proof}

We are ready to prove the main Theorem of the section.
\begin{thm}\label{thm:siegel}
    Assume ${\rm (R4)}$. Let $K$ a global function field and $\pp$ a place of $K$, there exists $N$ such that $\pi_\pp(L)\leq N$ for any finite separable extension $L/K$.
\end{thm}
\begin{proof}
    Choose algebras $A,B$ according to Proposition \ref{brauer}. Let $\mathcal{K}$ denote the class of finite separable extensions of $K$. Notice that by Proposition \ref{keyproplocal} and Proposition \ref{keyprop}  the class $\mathcal{K}$ and the diophantine family $D_{\pp,t_\pp,A,B}$ satisfy the assumptions of Proposition \ref{propdioph}. Finally, applying Proposition \ref{propdioph} we obtain the desired result.
\end{proof}

\section{Kummer extensions}
Let $p$ be an odd rational prime, let $q$ be a power of $p$ and let $n$ be natural number coprime with $p$. Let $K=\F_q(t)$ and $\O_K:=\F_q[t]$. Let $P_n^+$ denote the set of irreducible monic polynomials of $\O_K$ of degree a multiple of $n$. Suppose that there is a primitive $n$-th root of unity $\zeta_n\in K$. 
Let $\{ p_m(t): m\in\mathbb{N}\}$ be an enumeration of $P^+_n$. We define recursively a tower of fields as follows $K_0=K$ and $K_{m+1}:=K_m\left (\sqn{p_m(t)}\right ).$ Let $\O_{K_m}:={\rm ic}_{K_m}(\O_K)$ where ${\rm ic}_{K_m}(\O_K)$ denotes the integral closure of $\O_K$ in $K_m$. Finally, let $K_{\rm inf}:=\bigcup\limits^\infty_{m=1}K_m$.
\begin{proposition}
    For any $p(t)\in P_n^+$, $K\left (\sqn{p(t})\right )\subseteq \F_q((\frac{1}{t}))$. 
\end{proposition}
\begin{proof}
    Write $p(t)=t^{m}+a_{m-1}t^{m-1}+\dots+a_0$ with $a_{m-1},\dots,a_0\in \F_q$. Then $t^{-m}p(t)=1+\frac{a_{m-1}}{t}+\dots+\frac{a_0}{t^m}\in \F_q[[\frac{1}{t}]] $. Applying Hensel's Lemma to the valuation ring $\F_q[[\frac{1}{t}]]$ and to the polynomial $\alpha(X):=X^n-t^{-m}p(t)$. We obtain that $\alpha(X)$ has an $n$--root in $\F_q[[\frac{1}{t}]]$. Since $n|m$, then $\sqn{p(t)}\in \F_q((\frac{1}{t}))$ as required. 
\end{proof}
\begin{proposition}
   Let $p(t)\in P_n^+$. Let $F:=K\left (\sqn{p(t})\right )$, and $\O_F:={\rm ic}_F(\O_K)$. Then $\O_F=\O_K[\sqn{p(t)}]$ and the discriminant $\mathfrak{d}_{\O_F/\O_K}=\pp^{n-1}$ where $\pp$ denotes the zero of $p(t)$. 
\end{proposition}
\begin{proof}
Let $u=\sqn{p(t)}$. We claim that $1,\dots,u^{n-1}$ is an $\O_K$--integral basis. To see this, it is sufficient to prove that it is locally a basis for each prime $\qq$ different from the prime at infinity $\pp_\infty$.   Let $\Pp$ be a prime of $F$ not above $\pp$. Then $v_\Pp(\varphi'(u))=0$. Using  \cite[Corollary 3.5.11]{stichtenoth} we obtain that $1,u,\dots,u^{n-1}$ is an integral basis at all primes different from  $\pp$ and $\pp_\infty$. Using  \cite[Proposition 6.3.1]{stichtenoth} we obtain that $\pp$ is completely ramified and no other prime (including the prime at infinity) ramifies. Finally, applying \cite[Corollary 3.5.12]{stichtenoth} we have that $1,u,\dots,u^{n-1}$ is also an integral basis at $\pp$. Next we compute the discriminant. From the previous calculation and Dedekind different's Theorem, we have that  $\dif(\O_F/\O_K)=(n-1) \Pp$ where $\Pp$ denotes the only prime in $F$ below $\pp$. Thus $\mathfrak{d}_{\O_F/\O_K}=\pp^{n-1}$. This concludes the proof of the Proposition.\end{proof}

\begin{proposition}
    For each $N$, $\O_{K_{N+1}}=\O_{K_N}[\sqn{p_N(t)}]$.
\end{proposition}
\begin{proof}
    Let $F_m:=K(\sqn{p_m(T)})$ and $u_m:=\sqn{p_m(T)}$. We shall prove by induction on $N$ that:
    \begin{itemize}
        \item the fields $F_1,\dots,F_N$ are linearly disjoint (inside $\F_q((t))$),
     \item the discriminant $\mathfrak{d}_{\O_{K_N}/\O_K}=\prod\limits_{m=1}^N\mathfrak{d}_{\O_{F_m}/\O_K}^{n^{N-1}}$, and 
     \item the set $\{\prod\limits_{m=1}^N u_m^j: j=1,\dots,n-1\}$ is an $\O_K$--integral basis of $\O_{K_N}.$
     \end{itemize}
It follows from the induction hypothesis and the computation  of the discriminant in the previous Proposition that the discriminants $\mathfrak{d}_{\O_{K_{N}/\O_K}}$ and  $\mathfrak{d}_{\O_{K_{N+1}/\O_K}}$ are comaximal. To prove the first bullet point, it suffices to show that the minimal polynomial $\alpha$ of $u_N$ over $K$ is equal to the minimal   polynomial $\beta$ of $u_N$ over $F_{N+1}$. For this we need to show that the coefficients of $\beta$ belong $K$. Since ${K_N}$ is a Galois extension we have that the coefficients of $\beta$ belong to $K_N\cap F_N:=G$. Notice that $\mathfrak{d}_{\O_G/\O_K}=(1)$ as the discriminants are comaximal. Thus we have that $G$ is an unramified extension of $K$. By \cite[Corollary 6.21]{lenstra} $G$ is a constant extension of $K$. Since $G, K\subseteq \F_q((t))$, then $K=G$. Thus the fields are linearly disjoint. The rest of the proof is completely analogous to the proof of \cite[Theorem 4.26]{narkiewicz}.    
\end{proof}

For $x\in K_{\rm inf}$, let $\sigma_1,\dots,\sigma_n$ be the complete set of $K$-embeddings of $K(x)$ into $\F_q((\frac{1}{t})).$   Define $\|x\|_{\rm max}=\max\{\|\sigma_i(x)\|_\infty:i=1,\dots,n\}. $

\begin{proposition}
For any positive integer $N$. The set $\{x\in K_{\rm inf}: \|x\|_{\rm max}\leq N\}$ is finite.
\end{proposition}
\begin{proof}
    Let $x\in K_{m+1}$ be given. Set $u_m:=\sqn{p_m(t)}$. Then $x=a_0+a_1u_m+\dots+a_{n-1}u_m^{n-1}$ where $a_0,\dots,a_{n-1}\in \O_{F_m}.$ The conjugates of $x$ are given by $\sum\limits_{i=0}^{n-1} \sigma(a_i)\zeta_n^{ji} u_m^i $ where $\sigma\in \Gal(K_m/K)$ and $j=0,\dots,n-1$.
Using the relation $1+\zeta^i_n+\dots+\zeta_n^{i(n-1)}=0$ we obtain that 
$$
\sum\limits_{j=0}^{n-1}\sum\limits_{i=0}^{n-1} \sigma(a_i)\zeta_n^{ji} u_m^i=\sum\limits_{i=0}^{n-1} \left (\sigma(a_i) u_m^i\sum\limits_{j=0}^{n-1}\zeta_n^{ji}\right )=n\sigma(a_0) 
$$
It follows, from the strong triangle inequality,  that $\|\sigma(a_0)\|_\infty \leq \|x\|_{\rm max}$ for $\sigma \in \Gal(K_m/K)$. Hence, $\|a_0\|_{\rm max}\leq \|x\|_{\rm max}.$

   Using again the strong triangle inequality we have that $\|x-a_0\|_{\rm max}\leq  \|x\|_{\rm max}.$ Then we have $\|u_m^{-1}(x-a_0)\|_{\rm max}\leq \|x\|_{\rm max}q^{-\deg(p_m)/n}$. Applying the previous procedure to $y:=\zeta_n^{-1}u_m^{-1}(x-a_0)=a_1+a_2\zeta_nu_m+\dots+a_{n-1}\zeta_n^{n-2}u_m^{n-2}$. We obtain $\|a_1\|_{\rm max}\leq \|x\|_{\rm max}q^{-\deg(p_m)/n}$. Repeating the same argument $n$ times we obtain that $\|a_i\|_{\rm max}\leq \|x\|_{\rm max}q^{-i\deg(p_m)/n}$ for $i=0,\dots,n-1$. If $q^{\deg(p_m)}>N,$ then it follows that $a_i=0$ for $i=1,\dots, n-1.$ Since there are only finitely many polynomials satisfying the inequality $q^{\deg(p_m)}\leq N$, then there is an $M$ such that $\{x\in K_{\rm inf}: \|x\|_{\rm max}\leq N\}=\{x\in K_M: \|x\|_{\rm max}\leq N\}.$ Since the last set is finite we obtain what we wanted.  
       
\end{proof}

We are ready to prove the main Theorem of the section.
\begin{thm}\label{thm:kummer}
Assume (R4). 
    The $\mathcal{L}_{\rm ring,t}$--first order theory of $K_{\rm inf}$ is undecidable. 

\end{thm}
   \begin{proof}
       Since $[K_m: K]=n^m$ and the extensions $K_m/K$ are Galois, we can apply  \cite[Theorem A]{MSU}, 
       to the following set up $K, K_{\rm inf}, \mathfrak{S}_K=\{\infty\} $ and $\ell$ (any prime different from $p$ and coprime with $n$). This gives us   a $\mathcal{L}_{\rm ring,t}$--first-order definition of $\O_{K_{\rm inf}}$  without parameters. A standard reduction argument then shows that the undecidability of $K_{\rm inf}$ follows from the undecidability of $\O_{K_{\rm inf}}$.
         Using \cite[Theorem A and Lemma 2.8]{MSU}, we further reduce the problem  to construct a parametrized family of definable sets of $\O_{K_{\rm inf}}$,  that contains sets of arbitrarily large finite cardinalities. To see this, consider the sets $\{x\in K_{\rm inf}:\|t^{-N}x\|_{\rm max}\leq 1\}.$ This last set corresponds to $\{x\in K_{\rm inf}:\forall \Pp\in \mathcal{S}^*_{\pp_\infty}(K_{\rm inf}) (\|t^{-N}x\|_\Pp\leq 1)\ \}$. By Proposition 6.1 the prime at infinity $\pp_\infty$ splits completely in $K_m$ for $m\in \mathbb{N}.$ Therefore $\{x\in K_{\rm inf}:\|t^{-N}x\|_{\rm max}\leq 1\}$ is equal to the integral closure of the Kochen ring in $K_{\rm inf}$. Finally, using Theorem 5.8 we obtain that this family is existentially  $\mathcal{L}_{\rm ring,t}$--definable. This concludes the proof of the Theorem. 
   \end{proof}

\addcontentsline{toc}{section}{References}
	

		
	\medskip

\printbibliography

\noindent Carlos A. Mart\'inez--Ranero\\
Email: cmartinezr@udec.cl\\
Homepage: www2.udec.cl/~cmartinezr\\
\noindent Javier Utreras\\
Email: javierutreras@udec.cl \hspace{10pt} javutreras@gmail.com\\

\noindent Same address: \\
Universidad de Concepci\'on, Concepci\'on, Chile\\
Facultad de Ciencias F\'isicas y Matem\'aticas\\
Departamento de Matem\'atica\\

\end{document}